\documentclass[12pt]{amsart}
\usepackage{amssymb,amsmath,amsthm}
\numberwithin{equation}{section}
\usepackage{enumerate}
\newtheorem{thm}{Theorem}[section]
\newtheorem{lem}[thm]{Lemma}

\newtheorem{cor}[thm]{Corollary}
\newtheorem{prop}[thm]{Proposition}
\theoremstyle{definition}

\newtheorem{eg}[thm]{Example}

\newtheorem{ques}[thm]{Question}

\begin{document}
\title[Universal commutative operator algebras]{Universal commutative operator algebras and transfer function realizations of polynomials}
\author{Michael T. Jury}

\address{Department of Mathematics\\
  University of Florida\\
  Box 118105\\
  Gainesville, FL 32611-8105\\
  USA}

\email{mjury@ufl.edu}
\thanks{Research supported by NSF grant DMS 0701268.}

\date{\today}

\begin{abstract}To each finite-dimensional operator space $E$ is associated a commutative operator algebra $UC(E)$, so that $E$ embeds completely isometrically in $UC(E)$ and any completely contractive map from $E$ to bounded operators on Hilbert space extends uniquely to a completely contractive homomorphism out of $UC(E)$.  The unit ball of $UC(E)$ is characterized by a Nevanlinna factorization and transfer function realization.  Examples related to multivariable von Neumann inequalities are discussed.
\end{abstract}
\maketitle
%\tableofcontents
\section{Introduction}
Consider the algebra $\mathcal P_n=\mathbb C[z_1,\dots z_n]$ of
polynomials in $n$ variables with complex coefficients.  If $T=(T_1,
\dots T_n)$ is an $n$-tuple of bounded, commuting operators on a
Hilbert space $H$, then we can define a seminorm $\|\cdot \|_T$ on
$\mathcal P_n$ in the obvious way:
\begin{equation*}
\|p\|_T :=\|p(T)\|.
\end{equation*}
If now $\mathcal T$ is a collection of such $n$-tuples which is
\emph{separating} for $\mathcal P_n$, that is, $p(T)=0$ for all $T\in
\mathcal T$ if and only if $p=0$, then the supremum
\begin{equation}\label{E:Tnorm}
\|p\|:=\sup_{T\in\mathcal T} \|p(T)\|.
\end{equation}
defines a norm on $\mathcal P_n$, and the closure of $\mathcal P_n$
with respect to this norm is a Banach algebra.  Moreover, if $p$ is an
$m\times m$ matrix of polynomials (equivalently, a polynomial with
$m\times m$ matrix coefficients), we can similarly define
\begin{equation}\label{E:Tnorm_matrices}
\|p\|_n:=\sup_{T\in\mathcal T} \|p(T)\|.
\end{equation}
Explicitly, if we write $p$ in multi-index notation as $p(z)=\sum_{\bf
  n} A_{\bf n} z^{\bf n}$ with $A_{\bf n}\in M_{m\times
  m}(\mathbb C)$, then $\|p(T)\|$ denotes the norm of the operator
\begin{equation*}
\sum_{\bf n} A_{\bf n} \otimes T^{\bf n}
\end{equation*}
acting on $\mathbb C^m \otimes H$, where $T^{\bf n}$ has the obvious
meaning.  The completion of $\mathcal P_n$ in the norm (\ref{E:Tnorm})
(together with the system of matrix norms $\|\cdot \|_n$) thus
becomes an {\em operator algebra}.  While not every Banach algebra
norm on $\mathcal P_n$ can be obtained in this way, a number of norms
of this type arise naturally and have been extensively studied.  It will help to consider some examples.  

1)  Fix a nice domain $\Omega\subset \mathbb C^n$ (say, the unit ball)
and let $T$ range over the commuting normal operators with joint
spectrum in $\partial \Omega$.  Then by the spectral theorem (and the
maximum principle) the norm $\|p\|_{\mathcal T}$ is equal to the
supremum norm $\|p\|_\infty =\sup_{z\in\Omega}|p(z)|$.

2)  Another important example is the {\em universal} or {\em Agler norm}
$\|\cdot \|_u$.  Here $T$ ranges over all commuting $n$-tuples of
contractive operators on Hilbert space.  It is known that
$\|\cdot\|_u$ is equal to the supremum norm over the unit circle
$\mathbb T$ when $n=1$ (this follows from von Neumann's inequality),
and equal to the supremum norm over the $2$-torus $\mathbb T^2$ when
$n=2$ (Ando's inequality), but the analogous statements are false for
all $n\geq 3$; a counterexample was first given by Kaijser and
Varopoulos \cite{var}.

3)  Yet another well-known example comes from the \emph{row contractions};
these are the commuting $n$-tuples for which
\begin{equation*}
I-\sum_{j=1}^n T_j T_j^* \geq 0.
\end{equation*}
It is remarkable in this case that the supremum (\ref{E:Tnorm}) is
always attained on a single distinguished row contraction, namely the
\emph{$n$-shift} $S_1, \dots S_n$ where $S_j$ is the operator of
multiplication by the coordinate function $z_j$ on a certain Hilbert
space of holomorphic functions on the unit ball of $\mathbb C^n$.  The
resulting norm on a polynomial $p$ is the norm $p$ inherits by acting
as a multiplication operator on this space, called the {\em
  Drury-Arveson space} \cite{drury,arv,btv}, which is the reproducing
kernel Hilbert space on the unit ball of $\mathbb C^n$ with kernel
$k(z,w)=(1-\sum_j z_j\overline{w_j})^{-1}$.  It is known that this
norm is generically strictly greater than the supremum norm over the
ball, and in fact the two are inequivalent \cite{arv}.  (In the
previous example, it is not known whether the universal norm is
equivalent to the supremum norm over $\mathbb T^n$ when $n\geq 3$.)
\vskip.1in
Following Ambrozie and Timotin \cite{at} and Ball and Bolotnikov
\cite{bb} (and the more general approach of Mittal and Paulsen
\cite{mittal-paulsen}) the last two examples may be unified in the following way: consider
the $n\times n$-matrix valued function
\begin{equation*}
Q(z_1, \dots z_n) =\text{diag}(z_1, \dots z_n).
\end{equation*}
Then the operators $T_1, \dots T_n$ are all contractive if and only if
\begin{equation}
I-Q(T)Q(T)^* \geq 0.
\end{equation}
Row contractions are similarly characterized by the positivity of
$I-Q(T)Q(T)^*$ for the $1\times n$-matrix valued function
\begin{equation*}
Q(z)=(z_1, \dots z_n).
\end{equation*}
In general, then, fix an analytic $N\times M$ matrix-valued polynomial
$Q$ in $n$ variables and consider the domain
\begin{equation*}
\mathcal D_Q =\{ z\in \mathbb C^n : \|Q(z)\|<1\}.
\end{equation*}
(Examples (2) and (3) above give the unit polydisk $\mathbb D^n$ and the unit
ball $\mathbb B^n$ respectively.)  Now consider the class of commuting
operator $n$-tuples
\begin{equation*}
\mathcal T_Q =\{ T=(T_1,\dots T_n) : I-Q(T)Q(T)^*\geq 0\}.
\end{equation*}
This class of operators may be used to define an operator algebra norm
on the space of polynomials as in (\ref{E:Tnorm}).  Say a polynomial
lies in the \emph{Schur-Agler class} $\mathcal{SA_Q}$ if $\|p(T)\|\leq
1$ for all $T$ such that $I-Q(T)Q(T)^*\geq0$.  (It is possible for
different polynomials $Q$ to determine the same domain $\mathcal D_Q$
but distinct Schur-Agler classes $\mathcal{SA}_Q$; one of the
motivations of the present paper is to investigate these differences
in the case of linear $Q$.)  Among the main results of \cite{at} and
\cite{bb} is that the classes $\mathcal{SA}_Q$ are characterized by a
``Nevanlinna factorization,'' so named because it may be read as a
kind of generalization of a 1919 theorem of R. Nevanlinna.  This
theorem says that a function $f$ in the unit disk $\mathbb
D\subset\mathbb C$ is holomorphic and bounded by $1$ if and only if
the kernel $(1-f(z)\overline{f(w)})(1-z\overline{w})^{-1}$ is positive
semidefinite.  Equivalently, there exists a Hilbert space $H$ and a
holomorphic function $F:\mathbb D\to H$ such that
\begin{equation}\label{E:pickmatrix1}
\frac{1-f(z)\overline{f(w)}}{1-z\overline{w}}=F(z)F(w)^*
\end{equation}
which it will be helpful to rewrite as
\begin{equation}\label{E:pickmatrix2}
1-f(z)\overline{f(w)}=F(z)(1-z\overline{w})F(w)^*.
\end{equation}
A version of this theorem in the bidisk $\mathbb D^2$ was obtained by
Agler \cite{agmcc}, who showed that $f:\mathbb D^2\to \mathbb C$ is
holomorphic and bounded by $1$ if and only if there exists Hilbert
space $H$ and holomorphic functions $F_1,F_2:\mathbb D^2\to H$ such
that
\begin{equation}\label{E:aglerpick1}
1-f(z)\overline{f(w)} =F_1(z)(1-z_1\overline{w_1})F_1(w)^* + F_2(z)(1-z_2\overline{w_2})F_2(w)^*
\end{equation}
If we put $F=[F_1\ F_2]$ and define
\begin{equation}\label{E:aglerQdef}
Q(z)=\begin{pmatrix}z_1 & 0 \\ 0 & z_2\end{pmatrix}
\end{equation}
then (\ref{E:aglerpick1}) takes a form more reminiscent of
(\ref{E:pickmatrix2}):
\begin{equation}\label{E:aglerpick2}
1-f(z)\overline{f(w)} = F(z) \left[ I_H \otimes (I_2 - Q(z)Q(w)^*)\right]F(w)^*
\end{equation}                                               
In general, we have the following, which is special case of
\cite[Theorem 1.5]{bb}.  (We state the theorem only in the case of
polynomials, since it is really the operator algebra norm induced by
the operators $T$ that is of interest in the present paper.)

\begin{thm}\label{T:bb}
Let $Q$ and $\mathcal D_Q$ be as above, and let $p$ be a matrix-valued
analytic polynomial.  Then the following are equivalent:
\begin{itemize}
\item[{\bf 1)}]  {\textbf{\emph{Agler-Nevanlinna factorization.}}}  There exists a Hilbert space $K$, and an analytic function $F:\mathcal D_Q \to B(K,\mathbb C^N)$ such that
\begin{equation}
1-p(z)p(w)^* = F(z)\left[ I_K\otimes(I -Q(z)Q(w)^* )\right]F(w)^*
\end{equation}

\item[{\bf 2)}]  {\textbf{\emph{Transfer function realization.}}}  There exists a Hilbert space $K^\prime$, a unitary transformation $U: K^\prime\oplus \mathbb C^N \to K^\prime\oplus \mathbb C^N$ of the form
\begin{equation}
\begin{array}{cc}
		\	& 	\begin{array}{cc}
		 		K^\prime & \mathbb C^N
		 		\end{array}\\
		\begin{array}{c}
			K^\prime \\
			\mathbb C^N
		\end{array}  & 	\left( \begin{array}{ll}
						A & B \\
						C & D
				       \end{array} \right)
	\end{array}
\end{equation}
such that
\begin{equation}
p(z) = D + C(I-Q(z)A)^{-1}Q(z)B.
\end{equation}
\item[{\bf 3)}]  {\textbf{\emph{von Neumann inequality.}}}  $p$ lies in the (matrix-valued) Schur-Agler class $\mathcal{SA}_Q$, that is, for every commuting $n$-tuple $T$ such that $I-Q(T)Q(T)^*\geq 0$,
\begin{equation}
{\|p(T)\|}_{M_N\otimes B(K)} \leq 1.
\end{equation}
\end{itemize}
\end{thm}

The inequality in statement (3) always implies that $p$ is bounded by
$1$ in $\mathcal D_{\mathcal Q}$.  The converse holds in $\mathbb D$
(von Neumann's inequality) and in $\mathbb D^2$ (with $Q$ given by
(\ref{E:aglerQdef})) by Ando's theorem.  In all other cases the
converse is either false or an open problem.

The purpose of this paper is to prove an analog of the above result
where the single matrix-valued polynomial $Q$ is replaced by a family
of \emph{linear} maps $\sigma:\mathbb C^n\to B(H)$.  In this respect
there is some overlap with the more general results of
\cite{mittal-paulsen}, where general analytic $\sigma$ are considered,
though the point of view of the present paper is somewhat different.
In particular the linear maps $\sigma$ are exactly the maps that are
completely contractive with respect to a given $n$-dimensional
operator space $E$.  (We assume the reader is familiar with the
notions of operator spaces, completely contractive maps, etc.  The
books \cite{paulsen} and \cite{pisier} are excellent references.  The
facts and definitions we require are briefly reviewed in
Section~\ref{S:prelim}.)  The role of these operator spaces (and their
duals) is a central theme.  (In fact the reader who is familiar with
the results of \cite{at, bb, mittal-paulsen} may prefer to read
sections \ref{sec:universal} and \ref{S:examples} first.)

To state our main theorem, we introduce one bit of notation: if
$S=(S_1, \dots S_n)$ is an $n$-tuple of operators on a Hilbert space
$K$, write $\sigma_S$ for the map
\begin{equation}
\sigma_S(z):=\sum_{j=1}^n z_j S_j
\end{equation}
from $\mathbb C^n$ into $B(K)$.  (Evidently every linear map
$\sigma:\mathbb C^n\to B(K)$ has this form.)  Our main theorem, proved
in Section~\ref{S:main}, is the following:

\begin{thm}\label{T:main}  Let $E$ be a finite dimensional operator space, with underlying Banach space $V$, and let $\Omega\subset \mathbb C^n$ denote the open unit ball of $V$.   For each analytic $M_N$-valued polynomial $p$, the following are equivalent:
\begin{itemize}
\item[{\bf 1)}]  {\textbf{\emph{Agler-Nevanlinna factorization.}}}  There exists a Hilbert space $K$, a completely contractive map $\sigma:E\to B(K)$, and an analytic function $F:\Omega \to B(K,\mathbb C^N)$ such that
\begin{equation}\label{E:factorization}
1-p(z)p(w)^* = F(z)\left[ I_K -\sigma(z)\sigma(w)^* \right]F(w)^*
\end{equation}
for all $z,w\in\Omega$.
\item[{\bf 2)}]  {\textbf{\emph{Transfer function realization.}}}  There exists a Hilbert space $K^\prime$, a unitary transformation $U: K^\prime\oplus \mathbb C^N \to K^\prime\oplus \mathbb C^N$ of the form
\begin{equation}
\begin{array}{cc}
		\	& 	\begin{array}{cc}
		 		K^\prime & \mathbb C^N
		 		\end{array}\\
		\begin{array}{c}
			K^\prime \\
			\mathbb C^N
		\end{array}  & 	\left( \begin{array}{ll}
						A & B \\
						C & D
				       \end{array} \right)
	\end{array}
\end{equation}
and a completely contractive map $\sigma:E\to B(K^\prime)$ so that
\begin{equation}\label{E:transfer}
p(z) = D + C(I-\sigma(z)A)^{-1}\sigma(z)B.
\end{equation}
for all $z\in\Omega$.  
\item[{\bf 3)}]  {\textbf{\emph{von Neumann inequality.}}}  If $S$ is a commuting $n$-tuple in $B(K)$ and $\sigma_S$ is completely contractive for $E^*$, then 
\begin{equation}\label{E:vonneumann}
{\|p(S)\|}_{M_N\otimes B(K)} \leq 1.
\end{equation}
\end{itemize}
\end{thm} 

A quick observation: if the operators $S=(S_1,\dots S_n)$ in statement
(3) are commuting {\em matrices}, then the condition that $\sigma_S$
be completely contractive for $E^*$ is just the condition that $S$
belong to the unit ball of $E$.  This duality is described more fully
in the next section.  

We now let $UC(E^*)$ denote the completion of the polynomials in the
norm
\begin{equation}\label{eqn:universal-norm-def}
\|p\|_{UC(E^*)}:=\sup_S \|p(S)\|
\end{equation}
where the supremum is taken over all $S$ appearing in item 3 of
Theorem~\ref{T:main}, and let $UC(E^*)$ denote the resulting operator
algebra ($UC$ for ``Universal Commutative'').  It is easy to see that
the norm (\ref{eqn:universal-norm-def}) controls the supremum norm over
$\Omega$, and hence every element of $UC(E^*)$ is a continuous
function on the closure of $\Omega$ and analytic in $\Omega$.

These algebras $UC(E^*)$ are the {\em universal commutative operator
  algebras} of the title.  Indeed, it is evident that $UC(E^*)$ has
the following universal property: if $\sigma:E^*\to B(H)$ is any
completely contractive map with commutative range, then $\sigma$ has a
unique extension to a completely contractive homomorphism
$\hat{\sigma}:UC(E^*)\to B(H)$.  (The map $\sigma$ picks out a
commuting $n$-tuple $S$, and $\hat\sigma$ just evaluates on $S$.)
This is discussed further in Section~\ref{sec:universal}.

The results of this paper overlap with those of \cite{at,bb} only for
those operator spaces $E$ which can be embedded completely
isometrically in $B(H)$ for some \emph{finite-dimensional} Hilbert
space $H$.  However many finite-dimensional operator spaces of
interest do \emph{not} admit such an embedding.  Indeed among the most
interesting operator spaces in the present context are the so-called
\emph{maximal operator spaces} $MAX(V)$, which correspond to the {\em
  minimal $UC$-norms} discussed in Section~\ref{S:examples}.  (In
particular we show that the supremum norm on the tridisk $\mathbb D^3$
is not a $UC$-norm.)  Outside the exceptional cases of the $V=\ell^1$
and $V=\ell^\infty$ norms on $\mathbb C^2$, we know of no maximal
space which embeds in a matrix algebra (or even a nuclear C*-algebra;
in fact for every $n>16$ there exist $V$ with ${\rm dim}V=n$ such that
$MAX(V)$ cannot embed in a nuclear C*-algebra.  See
\cite[pp.340--341]{pisier}).

\section{Preliminaries}\label{S:prelim}
\subsection{Operator spaces and duality}
Let $\|\cdot\|$ be a norm on $\mathbb C^n$; write $V$ for the Banach
space $(\mathbb C^n,\|\cdot\|).$ Let $\Omega$ denote the open unit
ball of $V$:
\begin{equation*}
\Omega = \{z\in \mathbb C^n : \|z\|<1\}
\end{equation*}
Let $\langle \cdot, \cdot\rangle$ denote the standard \emph{symmetric} (not Hermitian)  inner product on $\mathbb C^n$:
\begin{equation*}
\langle z,w\rangle =\sum_{j=1}^n z_j w_j
\end{equation*}
We consider the dual space $V^*$ with respect to this pairing, and let
$\Omega^*$ denote the dual unit ball:
\begin{equation*}
\Omega^* = \{z\in \mathbb C^n : |\langle z, w\rangle|<1 \text{ for all } w\in\Omega\}
\end{equation*}

We will equip $V$ with various (concrete) operator space structures;
each of these is determined by an isometric mapping
\begin{equation*}
\varphi : V\to B(H)
\end{equation*}
where $H$ is a Hilbert space (of arbitrary dimension).  More explicitly, if we write vectors $z\in V$ in coordinate form with respect to the standard basis $e_1, \dots e_n$, we will write $T_j =\varphi(e_j)$ so that
\begin{equation*}
\varphi(z)=\varphi(\sum z_je_j) =\sum z_j T_j :=\langle z, T\rangle
\end{equation*}
The matrix norm structure on $E=(V,\varphi)$ is determined explicitly as follows:  if $A_1, \dots A_n$ are matrices (all of some fixed size $k\times l$) then 
\begin{equation*}
\|(A_1\ \dots A_n)\|_{\varphi} := \|\sum A_j\otimes T_j\|
\end{equation*}
where the latter norm is the standard one in $M_{kl}\otimes B(H)$, obtained by identifying $M_{kl}\otimes B(H)$ with $B(H^{l}, H^{k})$.  It will be convenient to write
\begin{equation*}
\langle A, T\rangle := \sum A_j\otimes T_j.
\end{equation*}

Given an operator space structure $E=(V,\varphi)$ and a linear map
$\psi:V\to B(K)$, we say $\psi$ is \emph{completely contractive} with
respect to $\varphi$ if
\begin{equation}\label{E:cc}
\|\sum A_j \otimes \psi(e_j)\|\leq \|\sum A_j\otimes \varphi(e_j)\|
\end{equation}
for all $n$-tuples of matrices $A=(A_1, \dots A_n)$.  The map $\psi$
will be called \emph{completely isometric} if equality holds in
(\ref{E:cc}) for all $A$.

An operator space structure $E$ over $V$ naturally determines a dual
operator space structure $E^*$ over $V^*$, by declaring
\begin{equation*}
E^* := CB(E, \mathbb C)
\end{equation*}
At the first matrix level, $M_1(E^*)$ is isometrically $V^*$.  The
$M_m(E^*)$-norm of an $m\times m$ matrix $A$ with entries from $V^*$
is then given by the $CB$ norm of the map from $V$ to $M_m(\mathbb C)$
induced by $A$.  In the case that $E$ is finite dimensional, the
duality can be described much more concretely in terms of a pairing
between completely contractive maps for $E$ and $E^*$.  It will be
helpful to work this out very explicitly: first we recall that, since
$V$ is identified with $\mathbb C^n$ as a vector space, by the
``canonical shuffle'' elements of $M_m(E)$ may be presented in one of
two ways: either as $m\times m$ matrices with entries from $V$,
\begin{equation*}
A=[ \vec{a}_{ij}]_{i,j=1}^m, \quad \vec{a}_{ij}=(a_{ij}^1, \dots a_{ij}^n)\in V,
\end{equation*}
 or as $n$-tuples of $m\times m$ matrices 
\begin{equation*}
A= [A_1 \dots A_n]  
\end{equation*}
where the $i,j$ entry of $A_k$ is $a_{ij}^k$.  In general we will
prefer the latter form.  In particular it will be desirable to
describe the matrix norms on $E^*$ using these expressions.  Fix an
element $A\in M_m(E^*)$.  Then $A$ induces a map from $V$ to
$M_m(\mathbb C)$ via
\begin{equation*}
A\cdot \vec{v} =[ \langle \vec{a}_{ij}, \vec{v}\rangle ]_{i,j=1}^m 
\end{equation*}
In turn, $A$ sends an element $B\in M_l(E)$ to the $ml\times ml$ matrix
\begin{equation*}
A\cdot B = {[ \langle \vec{a}_{ij}, \vec{b}_{pq}\rangle ]_{i,j=1}^m}_{p,q=1}^l
\end{equation*}
Using the definition of the symmetric pairing $\langle,
\cdot,\cdot\rangle$ this last matrix may be written as a sum
\begin{equation*}
{{\left[ \sum_{k=1}^n a_{ij}^k b_{pq}^k \right]}_{i,j=1}^m}_{p,q=1}^l
\end{equation*}
Now, for fixed $k$, the $ml\times ml$ matrix 
\begin{equation*}
{[a_{ij}^k b_{pq}^k ]_{i,j=1}^m}_{p,q=1}^l
\end{equation*}
may be canonically identified with the Kronecker tensor product
$A_k\otimes B_k$.  Thus, up to a canonical shuffle, the matrix $A\cdot
B$ is equal to
\begin{equation*}
\sum_{k=1}^n A_k\otimes B_k
\end{equation*}
Finally, by the definition of the matrix norms on $E^*$, we have that
the norm of $A$ in $M_m(E^*)$ is equal to
\begin{equation}\label{E:dual_norm_tensor}
\|A\|_{M_m(E^*)}:=\sup \| A\cdot B\|_{M_{ml}(\mathbb C)}=\sup{\left\|\sum_{k=1}^n A_k\otimes B_k\right\|}_{M_{ml}(\mathbb C)}
\end{equation}
where the supremum is taken over all $l\geq 1$ and all $B$ in the unit
ball of $M_l(E)$.  Similarly, we have for all $B\in M_l(E)$
\begin{equation}\label{E:E_norm_tensor}
\|B\|_{M_l(E)} =\sup{\left\|\sum_{k=1}^n A_k\otimes B_k\right\|}_{M_{ml}(\mathbb C)}
\end{equation}
where the supremum is taken over all $m\geq 1$ and all $A$ in the unit
ball of $M_m(E)^*$.

The above considerations extend naturally to the setting of completely
contractive maps.  Given an $n$-tuple of operators $T=(T_1, \dots
T_n)$ on a Hilbert space $B(H)$, define a linear map $\sigma_T:\mathbb
C^n\to B(H)$ by
\begin{equation*}
\sigma_T(z)=\sum_{j=1}^n z_j T_j.
\end{equation*}
\begin{prop}\label{P:dual_op_space}
Let $E$ be an $n$-dimensional operator space and let $E^*$ denote its
dual.  Given an $n$-tuple of operators $S=(S_1, \dots S_n)$, the map
$\sigma_S$ is completely contractive for $E^*$ if and only if
\begin{equation}\label{E:ST_pairing}
{\left\|{\sum_{j=1}^n S_j\otimes T_j}\right\|}_{min}\leq 1
\end{equation}  
for all $n$-tuples $T$ for which the map $\sigma_T$ is completely
contractive for $E$.
\end{prop}

\textbf{Remark:} Throughout this paper, the norm in expressions such
as (\ref{E:ST_pairing}) is understood to be the \emph{minimal} tensor
norm, that is, if $S$ and $T$ act on Hilbert spaces $H$ and $K$
respectively, the norm of the sum $\sum_{j=1}^n S_j\otimes T_j$ is its
norm as an operator on $H\otimes K$.  We will henceforth omit the
\emph{min} subscript.

\begin{proof}
Suppose $\sigma_S$ is completely contractive for $E^*$.  Then for all
$A=(A_1,\dots A_n)$ in the unit ball of $M_l(E^*)$,
\begin{equation*}
\left\|\sum_{k=1}^n S_k \otimes A_k\right\|\leq 1
\end{equation*}
Now let $T$ be an $n$-tuple of operators on a (separable) Hilbert
space $H$, such that $\sigma_T$ is completely contractive for $E$.
Fix an orthonormal basis for $H$ and let $P_k$ be the projection onto
the span of the first $k$ basis vectors.  Define an $n$-tuple of
$k\times k$ matrices
\begin{equation*}
A_j^k =P_kT_jP_k.
\end{equation*}
(The matrix of $A_j^k$ is written with respect to the fixed basis of
$H$.)  We claim that
\begin{equation*}
A^k=(A_1^k, \dots A_n^k)
\end{equation*}
belongs to the unit ball of $M_k(E^*)$.  To see this, by
(\ref{E:dual_norm_tensor}) it suffices to prove that
\begin{equation*}
\left\|\sum_{j=1}^n A_j^k\otimes B_j  \right\|\leq 1
\end{equation*}
for all $B=(B_1, \dots B_n)$ in the unit ball of $M_l(E)$, for all
$l$.  But since $\sigma_T$ is completely contractive for $E$, the map
$\sigma_k:=P_k\sigma_T P_k$ is as well, and we have
\begin{align*}
\left\|\sum_{j=1}^n B_j\otimes A_j^k  \right\|&= \left\|(I\otimes P_k)\left(\sum_{j=1}^n B_j\otimes T_j\right) (I\otimes P_k) \right\|\\
&\leq \left\|\sum_{j=1}^n B_j\otimes T_j  \right\|\\
&\leq 1
\end{align*}
since $\sigma_T$ is completely contractive.  This proves the claim.

Now, by the hypothesis that $\sigma_S$ is completely contractive for
$E^*$,
\begin{equation*}
\left\|\sum_{j=1}^n S_j\otimes A_j^k  \right\|\leq 1
\end{equation*}
for all $k$, but since 
\begin{equation*}
\sum_{j=1}^n S_j\otimes A_j^k \to \sum_{j=1}^n S_j\otimes T_j 
\end{equation*}
in the strong operator topology, we have $\|\sum_{j=1}^n S_j\otimes
T_j\|\leq 1 $, as desired.

For the converse, suppose that 
\begin{equation*}
\left\|\sum_{j=1}^n S_j\otimes T_j \right\|\leq 1
\end{equation*}
for all $T$ such that $\sigma_T$ is completely contractive for $E$.
By (\ref{E:dual_norm_tensor}), the map $\sigma_A$ is completely
contractive for $E$ whenever $A$ is an $n$-tuple of matrices in the
unit ball of $M_m(E^*)$.  It is then immediate that $\sigma_S$ is
completely contractive for $E^*$.
\end{proof}

\subsection{Factorization of positive semidefinite functions}
Let $\Omega$ be a set and $K$ a Hilbert space.  Following \cite{dmm},
a function $\Gamma:\Omega\times\Omega\to B(K)^*$ is called
\emph{positive semidefinite} if
\begin{equation}\label{E:semidef-gamma-scalar}
\sum_{z,w\in\Lambda}\Gamma(z,w)(f(z)f(w)^*) \geq 0
\end{equation}
for every finite subset $\Lambda\subset \Omega$ and every function
$f:\Lambda\to B(K)$.  This definition may be naturally extended if we
replace $B(K)^*$ with $B(B(K), M_N(\mathbb C))$: then $\Gamma:
\Omega\times\Omega \to B(B(K),M_N)$ is positive semidefinite if and
only if
\begin{equation}\label{E:semidef-gamma-matrix}
\sum_{z,w\in\Lambda}v(z)\Gamma(z,w)(f(z)f(w)^*)v(w)^* \geq 0
\end{equation}
for all finite $\Lambda \subset \Omega$, and all functions
$f:\Lambda\to B(K), v:\Lambda\to \mathbb C^N$.  In the scalar case,
the following lemma reduces to \cite[Proposition 4.1]{dmm}.  The proof
of the matrix-valued version stated here is entirely analogous and is
omitted.
\begin{lem}\label{L:mad-sam}
A function $\Gamma:\Omega\times\Omega \to B(B(K),M_N)$ is positive
semidefinite if and only if there exists a Hilbert space $H$ and a
function $G:\Omega\to B(B(K), B(H,\mathbb C^N))$ such that
\begin{equation}\label{E:mad-sam-factor}
\Gamma(z,w)(ab^*) =(G(z)[a])(G(w)[b])^*
\end{equation}
for all $a, b\in B(K)$.  
\end{lem}

\begin{lem}\label{L:mad-sam-main}  Suppose $E$ is a finite dimensional operator space, $\psi:E\to B(K)$ is completely contractive map, and $\Gamma:\Omega\times\Omega \to B(B(K),M_N)$ is a positive semidefinite function.  Then there exists a Hilbert space $H$, a completely contractive map $\sigma:E\to B(H)$ and a function $F:\Omega\to B(H, \mathbb C^N)$ such that
\begin{equation*}
\Gamma(z,w)[I_K - \psi(z)\psi(w)^*] =
F(z)(I_H-\sigma(z)\sigma(w)^*)F(w)^*.
\end{equation*}
\end{lem}
\begin{proof}  Given the completely isometric map $\psi:E\to B(K)$, let $\mathcal A$ denote the unital C*-subalgebra of $B(K)$ generated by the operators $\{\psi(z):z\in\Omega\}$.  Choose $G$ to factor $\Gamma$ as in Lemma~\ref{L:mad-sam}.  Now, in the factorization (\ref{E:mad-sam-factor}), let $H^\prime$ denote the subspace of $H$ spanned by vectors of the form $(G(w)[a])^*v$ for $w\in\Omega, a\in \mathcal A, v\in\mathbb C^N$.  We then obtain a ``right regular representation'' $\pi : \mathcal A\to B(H^\prime)$ by defining
\begin{equation}\label{E:leftrep}
\pi(a)^*(G(w)[b])^*v = (G(w)[ba])^*v
\end{equation} 
It is straightforward to check that $\pi$ is a $*$-homomorphism:
linearity is evident, and for all $x,y\in \mathcal A$ we have
\begin{align}
\pi(xy)^*(G(w)[b])^*v &= (G(w)[bxy])^*v\\
                      &= \pi(y)^*(G(w)[bx])^*v \\
                      &= \pi(y)^*\pi(x)^*(G(w)[b])^*v
\end{align}
so $\pi$ is multiplicative.  Similarly
\begin{align}
u^*G(z)[a]\pi(x)^*(G(w)[b])^*v &= u^*G(z)[a](G(w)[bx])^*v \\
                                 &= u^*\Gamma(z,w)[ax^*b^*]v \\
                                 &= u^*G(z)[ax^*](G(w)[b])^*v \\
                                 &= u^*G(z)[a]\pi(x^*)(G(w)[b])^*v
\end{align}
so $\pi(x^*)=\pi(x)^*$.  Now define $\sigma(z):= \pi(\psi(z))$.  It is
evident that $\sigma$ is a completely contractive map from $E$ to
$B(H^\prime)$.  It follows from (\ref{E:leftrep}) and the definition
of $\mathcal A$ that
\begin{equation}\label{E:rho_formula}
\sigma(w)^*(G(w)[b])^*v =(G(w)[b\psi(w)])^*v
\end{equation}
for all $z\in\Omega$, $b\in \mathcal A$ and $v\in\mathbb C^N$.  

Now define $H=K^\prime$ and $F(z):=G(z)[I_K]$.   It follows from  Lemma~\ref{L:mad-sam} and Equation~\ref{E:rho_formula} that
\begin{equation*}
\Gamma(z,w)[I_K] = G(z)[I_K](G(w)[I_K])^* =F(z)F(w)^*
\end{equation*}
and
\begin{align*}
\Gamma(z,w)[\psi(z)\psi(w)^*] &= (G(z)[I_K\psi(z)])(G(w)[I_K\psi(w)])^* \\
                 &=F(z)\sigma(z)\sigma(w)^*F(w)^*,
\end{align*}
and thus
\begin{equation*}
\Gamma(z,w)( I_K- \psi(z)\psi(w)^*) = F(z)F(w)^* - F(z)\sigma(z)\sigma(w)^*F(w)^*
\end{equation*}
as desired.
\end{proof}

%It will follow from Theorem~\ref{T:main} that this condition defines the unit ball of a norm on the space of polynomials.  Norms which arise in this way will be called \emph{Agler-Nevanlinna norms}(?).   We think of the maps $\psi$ as ``operator-valued test functions;'' one may take a different point of view and think of (\ref{E:agler}) as one way to generalize Nevanlinna's charachterization of holomorphic functions bounded by $1$ in the unit disk $\mathbb D\subset \mathbb C$.  
\section{Main Theorem}\label{S:main}
The proof of each implication in Theorem~\ref{T:main} is handled in a separate subsection.  
\subsection{1 implies 2}
\begin{proof} 
This is a standard application of the ``lurking isometry'' technique.  Rearrange (\ref{E:factorization}) to obtain
\begin{equation}\label{E:plusses}
1+ F(z)\sigma(z)\sigma(w)^*F(w)^* =p(z)p(w)^* + F(z)F(w)^*
\end{equation}
Define subspaces $\mathcal M, \mathcal N \subset H^\prime\oplus \mathbb C^N$ by
\begin{gather*}
\mathcal M=\text{span}\left\{ \begin{pmatrix} \sigma(w)^*F(w)^*x \\
                                             x\end{pmatrix} : w\in\Omega, \ x\in\mathbb C^N \right\} \\
 \mathcal N=\text{span}\left\{ \begin{pmatrix} F(w)^*x \\
   p(w)^*x\end{pmatrix} : w\in\Omega, \  x\in\mathbb C^N \right\}.
\end{gather*}   
The equation (\ref{E:plusses}) then implies the existence of an isometry $U^*:\mathcal M\to \mathcal N$ such that
$$
U^* \begin{pmatrix} \sigma(w)^*F(w)^*x \\
                  x\end{pmatrix} = \begin{pmatrix} F(w)^*x \\
                                                   p(w)^*x\end{pmatrix}
$$
for all $w\in\Omega$ and $x\in\mathbb C^N$.  Enlarging $H^\prime$ to a space $H^{\prime\prime}$ if necessary, we may extend $U^*$ to a unitary (still denoted $U^*$) taking $H^{\prime\prime} \oplus \mathbb C^N$ to itself.  We also regard $\sigma$ as taking $\mathbb C^N$ into $B(H^{\prime\prime})$, by declaring $\sigma(w)x$ to be $0$ for all $w\in\mathbb C^n$ and all $x\in H^{\prime\prime} \ominus H^\prime$.  (Note this extended $\sigma$ is still completely contractive.)  We now write the action of $U^*$ as a unitary colligation
\begin{equation*}
\begin{pmatrix}  A^*  & C^* \\
                 B^* & D^* \end{pmatrix}  \begin{pmatrix} \sigma(w)^*F(w)^*x \\
                  x\end{pmatrix} = \begin{pmatrix} F(w)^*x \\
                                                   p(w)^*x\end{pmatrix}
\end{equation*}
This corresponds to the linear system
\begin{gather}
A^*\sigma(w)^*F(w)^* + C^* = F(w)^* \\
B^*\sigma(w)^*F(w)^* + D^* = p(w)^*
\end{gather}
This system may be solved to obtain
\begin{equation*}
p(z) = D + C(I-\sigma(z)A)^{-1}\sigma(z)B 
\end{equation*}
for all $z\in \Omega$.  (Note that $(I-\sigma(z)A)$ is invertible, since $\|A\|\leq 1$ and $\|\sigma(z)\|\leq \|z\|_V <1$ for all $z\in \Omega$.)
\end{proof}

\subsection{2 implies 3}

\begin{proof}  Write
$$
\sigma(z)=\sum_{j=1}^n z_j T_j.
$$
Let $S=(S_1, \dots S_n)$ induce a completely contractive map $\sigma_S$ of $E^*$ on $B(L)$.  Then by Proposition~\ref{P:dual_op_space}, 
$$
\| \sum_{j=1}^n S_j\otimes T_j\|\leq 1. 
$$
Given the unitary colligation $U$, let $\tilde A =I_L\otimes A$, $\tilde B =I_L \otimes B$, etc.  Fix $0<r<1$; and observe that
\begin{equation}\label{E:transfer_on_S}
p(rS)=\tilde D +\tilde C \langle rS, T\rangle(I-\tilde A \langle rS, T\rangle )^{-1}  \tilde B.
\end{equation}
Since $\|rS\|<1$ and the $S_j$ commute, the right-hand side of (\ref{E:transfer_on_S}) may be expanded in a norm-convergent power series in the $S_j$.  Using (\ref{E:transfer}), we may also expand the left-hand side in the $S_j$, by first expanding $p$ and then substituting $rS$.  The equality then follows by matching coefficients.  
It is now easy to verify that 
\begin{equation}\label{eqn:prpos}
I-p(rS)^*p(rS) \geq 0
\end{equation}
for all $r<1$ and hence $I-p(S)^*p(S) \geq 0$ by letting $r\to 1$.  To prove (\ref{eqn:prpos}), let $A, B, C, D$ be any unitary colligation and $X$ any operator with $\|X\|<1$.  Then if we define
$$
Q= D +C X(I-AX)^{-1} B
$$
a well-known calculation shows that 
$$
I -Q^*Q = B^*(I-AX)^{-1*} (I-X^* X) (I-AX)^{-1} B\geq 0. 
$$
Taking $X=\langle rS, T\rangle$ and $Q=p(rS)$ proves the claim. 
\end{proof}

\subsection{3 implies 1}

\begin{proof}  This is the most involved part of the proof; the argument will be broken into several sub-arguments.  We will first show that, given any finite set $\Lambda\subset \Omega$, there exist $F$ and $\psi$ so that (\ref{E:factorization}) holds for all $z,w\in\Lambda$.  (This constitutes the bulk of the proof.)  We then conclude that such a factorization is valid on all of $\Omega$ via a compactness argument (in particular, by appeal to Kurosh's theorem).  

So, fix a finite set $\Lambda =\{\lambda_1,\dots \lambda_k\}\subset\Omega$.  Consider the cone $\mathcal C$ of $kN\times kN$ Hermitian matrices which can be written in the form
\begin{equation}\label{E:conedef}
A_{ij}=\left[ F(\lambda_i)\left(1-\psi(\lambda_i)\psi(\lambda_j)^* \right)F(\lambda_j)^*  \right]_{ij}
\end{equation}
where $F$ is a function from $\Lambda$ to a Hilbert space $B(K,\mathbb C^N)$ and $\psi$ is a completely contractive map of $E$ into $B(K)$.  It is easy to see that $\mathcal C$ contains all positive semidefinite matrices:  if $A$ is positive semidefinite we may factor it as $A_{ij}=F(\lambda_i)F(\lambda_j)^*$ and take $\psi=0$.  Moreover, observe that for all $A\in\mathcal C$, the Hilbert space $K$ in the above map can be taken to be a \emph{fixed} space of finite dimension at most $2kN$. To see this, note that the factorization that appears in the right hand side of \ref{E:conedef} takes place in the subspace of $K$ given by
\begin{equation*}
\text{span } \{ F(\lambda_i)^*x, \psi(\lambda_i)^*F(\lambda_i)^*x : i=1,\dots k, \ x\in\mathbb C^N\}
\end{equation*}

We now suppose that the $kN\times kN$ Hermitian matrix 
$$
P_{ij}= I_N-p(\lambda_i)p(\lambda_j)^*
$$
does \emph{not} belong to $\mathcal C$.  Our first claim is the following:

\emph{Claim 1:}  $\mathcal C$ is closed.  
  
It follows that there exists a real linear functional $L: M_{kN}^{sa}(\mathbb C)\to \mathbb R$ such that $L(\mathcal C)\geq 0$ but $L(P)<0$.  We extend $L$ to a complex linear functional on all of $M_{kN}(\mathbb C)$ (still denoted $L$) in the standard way.  Using $L$ we construct a pre-Hilbert space:  for functions $F, G:\Lambda\to B(K,\mathbb C^N)$, define
$$
\langle F, G\rangle_L :=L([F(\lambda_i)G(\lambda_j)^*])
$$
Since $L$ is positive on $\mathcal C$ and $\mathcal C$ contains all positive matrices, it follows that $\langle \cdot, \cdot\rangle_L$ is positive semidefinite.  Denote by $\mathcal H$ the resulting pre-Hilbert space.  We next construct an $n$-tuple of operators on $\mathcal H$.  First, if $Q:\Lambda\to B(K)$ is any function, we can define a ``right multiplication operator'' $M_Q$ on $\mathcal H$ via
\begin{equation*}
(M_Q F)(\lambda)=F(\lambda)Q(\lambda)
\end{equation*} 
(In fact, the only $Q$ we need will be scalar multiples of the identity, but it will be helpful to think of this scalar multiplication as occurring on the right rather than the left.)  Now, for each $\lambda_i\in\Lambda$, write its coordinates as
$$
\lambda_i=(\lambda_i^1, \dots \lambda_i^n)
$$
and define operators $S_k:\mathcal H\to \mathcal H$ by
$$
(S_kF)(\lambda_i) :=M_{\lambda^k I}F(\lambda_i)=F(\lambda_i)\lambda_i^k 
$$
We will construct from these operators a completely contractive map of the operator space $E^*$:  

\emph{Claim 2:}  If $\mathcal E$ is any Hilbert space and 
$$
\psi(z) =\sum_{k=1}^n z_k T_k 
$$
is any completely contractive map from $E$ to $B(\mathcal E)$, then the operator
$$
I-\left(\sum_{k=1}^n S_k\otimes T_k\right)^*\left(\sum_{k=1}^n S_k\otimes T_k\right)
$$
is nonnegative on  $\mathcal H\otimes \mathcal E$.

From this claim it follows easily that

\emph{Claim 3:}  If $\langle F, F\rangle_L =0$ then $\langle S_kF, S_kF\rangle_L=0$ for all $k=1,\dots n$.  

We may now construct a Hilbert space from $\mathcal H$ as usual, by passing to the quotient by the space of null vectors and completing; denote the resulting Hilbert space $\mathcal H^\prime$.    Claims 2 and 3 show that the operators $S_k$ pass to well-defined, bounded operators on $\mathcal H^\prime$, which will still denote $S_k$.  It is also immediate from Claim 2 that 
$$
\|\sum_{k=1}^n S_k \otimes T_k\|\leq 1
$$
whenever $\psi(z)=\sum z_k T_k$ is completely contractive for $E$; thus by Proposition~\ref{P:dual_op_space}, the map
$$
\varphi(z) =\sum_{k=1}^n z_k S_k
$$  
is completely contractive for $E^*$.  The proof that (\ref{E:factorization}) is valid on finite sets will now be complete if we can show that $1-p(S)^*p(S)$ is \emph{not} positive on $\mathcal H^\prime$.  Let $J$ denote the $kN\times kN$ which has the $N\times N$ identity matrix $I_N$ in the $i,j$ block for all $i,j=1, \dots k$.  The matrix $J$ may be factored as $G(\lambda_i)G(\lambda_j)^*$ where $G(\lambda_i)=I_N$ for all $i$.  Then 
\begin{align*}
\langle (I-p(S)^*p(S))G,G\rangle_{\mathcal H^\prime} &= \langle (I-p(S)^*p(S))G, G\rangle_L \\
&= L([G(\lambda_i)(I_N-p(\lambda_i)p(\lambda_j)^*)G(\lambda_j)^*]) \\
&= L(I_N-p(\lambda_i)^*p(\lambda_j)) \\
&< 0. 
\end{align*}
We have now proved the existence of the factorization on finite sets, modulo the proofs of the claims, which are now provided.  After these, the factorizations on finite sets will be pieced together, and the proof will be finished.

\emph{Proof of Claim 1:}  To see that $\mathcal C$ is closed we appeal again to the lurking isometry technique.  So, suppose $X\in\mathcal C$.  Since $X$ is Hermitian, by the spectral theorem we may write $X$ as a difference of two positive matrices 
$$
X= P-N
$$
with $\|P\|\leq \|X\|, \|N\|\leq \|X\|$.  Now factor $P$ and $N$ as Grammians:
$$
P_{ij}=\langle p_j , p_i \rangle , \qquad N_{ij}=\langle n_j, n_i\rangle
$$
There exist $F$ and $\psi$ so that
\begin{equation}\label{E:x_eqn}
X_{ij}=\langle p_j , p_i \rangle - \langle n_j, n_i\rangle =F(\lambda_j)^* (1-\psi(\lambda_j)^*\psi(\lambda_i))F(\lambda_i)
\end{equation}
As before, the lurking isometry argument leads to the equation
$$
F(\lambda_i) =(I-A\psi(\lambda_i))^{-1}Bp_i
$$
where $A, B$ belong to a unitary colligation.  Now, $\|A\psi(\lambda_i)\|\leq\|\lambda_i\|_V<1$ and $\|p_i\|\leq \|X\|$ for all $i$, and so 
\begin{equation}\label{E:key_closedness_estimate}
\|F(\lambda_i)\|\leq (1-\|\lambda_i\|_V)^{-1} \|X\|
\end{equation}
for all $i$.  

Let $X_n$ be a sequence in $\mathcal C$ and suppose $X_n\to X$.  For each $n$ we obtain $F_n, \psi_n$ so that (\ref{E:x_eqn}) holds.  By (\ref{E:key_closedness_estimate}) the functions $F_n$ are uniformly bounded, and hence admit a subsequence converging to some $F:\Lambda\to K$.  Since the $\psi_n$ are also uniformly bounded, passing to a further subsequence if necessary, we may assume that $\psi_n\to \psi$ pointwise in norm for some completely contractive $\psi$. (The fact that $\psi(z)$ acts on a finite-dimensional space is used here.)   It follows that this $F$ and $\psi$ factor $X$ as in (\ref{E:x_eqn}), and hence $X\in \mathcal C$.

\emph{Proof of Claim 2:}  Let $F_1, \dots F_d:\Lambda\to B(K,\mathbb C^N)$ and let $e_1, \dots e_d$ be an orthonormal set in $\mathcal E$.    To prove Claim 2 it suffices to show that 
\begin{equation}\label{E:claim2_main}
\left\langle \left(I-\left(\sum_{k=1}^n S_k\otimes T_k\right)^*\left(\sum_{k=1}^n S_k\otimes T_k\right)\right) (\sum F_l\otimes e_l), (\sum F_m \otimes e_m)\right\rangle_{\mathcal H \otimes \mathcal E}
\end{equation}
is positive; this will be the case because this is in fact may be written as the functional $L$ applied to an $kN\times kN$ matrix lying in the cone $\mathcal C$.  To see this, let us write $\tilde T_k$ for the operator $I_K\otimes T_k$ on $K\otimes \mathcal E$, and $\tilde F(\lambda_i) = \sum F_l(\lambda_i)\otimes e_l$.  By the definition of $S$ we have
\begin{align}
\sum_k(S_k\otimes T_k)(F_l\otimes e_l)(\lambda_i)&=\sum_k  F_l(\lambda_i)\lambda_i^k\otimes T_k e_l  \\
&= \left(\sum_k \lambda_i^k \tilde T_k \right) (F_l(\lambda_i)\otimes e_l)\\
&= \langle \lambda_i, \tilde T\rangle (F_l(\lambda_i)\otimes e_l) \\
&= \langle \lambda_i, \tilde T\rangle \tilde F(\lambda_i)
\end{align}
Now, using the fact that $\{e_l\}$ is orthonormal,
\begin{align}
\left\langle \sum F_l\otimes e_l, \sum F_m\otimes e_m \right\rangle_{\mathcal H\otimes \mathcal E} &= L\left(\sum F_l(\lambda_i)F_l(\lambda_j)^*\right) \\
&= L(\tilde F(\lambda_i)\tilde F(\lambda_j)^*)
\end{align}
Combining the above calculations, we find that (\ref{E:claim2_main}) may be written as 
\begin{equation}\label{E:claim2_second}
L\left(\tilde F(\lambda_i)[1-\langle \lambda_i, \tilde T\rangle \langle \lambda_j, \tilde T\rangle^*] \tilde F(\lambda_j)^*\right)
\end{equation}
Since the map $z\to \langle z, T\rangle$ is completely contractive for $E$, the map obtained by replacing $T$ with $\tilde T$ is as well.  It follows that the argument of $L$ in (\ref{E:claim2_second}) belongs to $\mathcal C$, and hence (\ref{E:claim2_main}) is positive, as desired.  

\emph{Proof of Claim 3:}  Trivially, there exists a real number $\alpha >0$ such that, for each $k=1, \dots n$, the map
\begin{equation*}
\sigma(z)= \alpha z_k
\end{equation*}
is a completely contractive map of $E$.  Applying Claim 2 to this map (that is, taking $T_k=\alpha$, $T_j=0$ for $j\neq k$) we get 
\begin{equation*}
I-\alpha^2 S_k^* S_k\geq 0
\end{equation*}
for each $k$.  Thus the operators $S_k$ are bounded on $\mathcal H$, so in particular $\langle S_k F, S_k F\rangle_L =0$ whenever $\langle F, F\rangle_L=0$.  

 We proved that a factorization (\ref{E:factorization}) exists on every finite subset $\Lambda \subset \Omega$.  The extension to all of $\Omega$ is accomplished via a routine application of Kurosh's theorem.  For each finite set $\Lambda\subset \Omega$ fix a factorization (\ref{E:factorization}).  Let $H_\Lambda$ be the Hilbert space on which $\psi$ acts.  Put $H:=\bigoplus_\Lambda H_\Lambda$ and $\psi :=\bigoplus_\Lambda \psi_\Lambda$.  Now for each $\Lambda$ let $\Phi_\Lambda$ be the set of all positive semidefinite functions $\Gamma_\Lambda :\Lambda\times\Lambda\to B(B(H),M_N)$ such that 
\begin{equation}\label{E:kurosh_factor}
1-p(z)p(w)^* = \Gamma_\Lambda(z,w)[I_H -\psi(z) \psi(w)^*]
\end{equation}
for all $z,w\in\Lambda$.  Each $\Phi_\Lambda$ is nonempty, since it contains
\begin{equation}
\Gamma_\Lambda(z,w)[a] =  F(z) P_\Lambda a P_\Lambda F(w)^*
\end{equation}
where $P_\Lambda:H\to H_\Lambda$ is the orthogonal projection.  By identifying $B(B(K),M_N))$ with $M_N(B(K)^*)$, the former space inherits the topology of entrywise weak-* convergence.  The set of functions from $\Lambda\times\Lambda$ to $B(B(K),M_N)$ may then be endowed with the topology of pointwise convergence in this topology on $B(B(K),M_N)$ (in other words, the ``pointwise entrywise weak-* topology'').  The sets $\Phi_\Lambda$ are then compact in this topology; this follows from the boundedness argument in the proof of Claim 1.  It is evident that restriction induces a continuous map $\pi_{\alpha\beta}:\Phi_\alpha\to \Phi_\beta$ when $\beta\subset \alpha$, so by Kurosh's theorem there exists a positive semidefinite $\Gamma$ which satisfies (\ref{E:kurosh_factor}) for all $z,w\in\Omega$.  Finally, applying Lemma~\ref{L:mad-sam-main} to this $\Gamma$ and $\psi$ finishes the proof.
\end{proof}

\section{Universality of $UC(E)$}\label{sec:universal}
In this section, to unclutter the notation a bit, we reverse the roles of $E$ and $E^*$ (which is harmless, since finite-dimensional operator spaces are reflexive), and consider the operator algebras $UC(E)$.  So 
\begin{equation}\label{eqn:univ-section-normdef}
\|p\|_{UC(E)}=\sup_S\{ \|p(S)\| \}
\end{equation}
the supremum taken over commuting $n$-tuples $S$ such that $\sigma_S:E\to B(K)$ is completely contractive.  As noted earlier, if the $S_j$ are matrices, then this is just the condition that $S$ lies in the unit ball of $E^*$.  

Pisier \cite[Chapter 6]{pisier} introduces the \emph{universal (unital) operator algebra} associated to an operator space $E$; this algebra is denoted $OA_u(E)$.  We will not require an explicit construction of $OA_u(E)$ here, only that $OA_u(E)$ has the following properties:
\begin{prop}\label{P:oauE} The following properties characterize $OA_u(E)$:
\begin{enumerate}  
\item There exists a canonical completely isometric embedding 
\[ 
\iota:E\to OA_u(E).
\]
\item If $\sigma:E\to B(H)$ is completely contractive, there exists a \emph{unique} completely contractive unital homomorphism $\hat{\sigma}:OA_u(E)\to B(H)$ extending $\sigma$, i.e. so that $\hat{\sigma}(\iota(x))=\sigma(x)$ for all $x\in E$.
\end{enumerate}
\end{prop} 

Similarly, the algebras $UC(E)$ are ``universal'' among commutative operator algebras containing $E$ completely contractively; in particular we have:

\begin{prop}\label{P:universalSA}  Let $E$ be a finite-dimensional operator space.  
\begin{enumerate}
\item  There exists a canonical completely isometric embedding 
\[
\iota:E\to UC(E).
\]
\item  If $\sigma:E\to B(H)$ is a completely contractive map with commutative range, then there exists a unique completely contractive unital homomorphism $\hat{\sigma}:UC(E)\to B(H)$ extending $\sigma$.  
\end{enumerate}
\end{prop}
\begin{proof}   
Everything is more or less immediate.  For the embedding of $E$ into
$UC(E)$, since the vector space underlying $E$ is just $\mathbb C^n$
we let the map $\iota$ send the point $a=(a_1,\dots a_n)$ to the
linear polynomial $p(z)=\sum a_jz_j$.  That this embedding is
completely isometric is immediate from the definition of the $UC(E)$
norms and the duality described in Section~\ref{S:prelim}.  For the
extension property, the map $\sigma$ has the form $\sigma(a)=\sum a_j
S_j$ for commuting $S_j$'s, and thus by definition
$\hat{\sigma}(p):=p(S)$ works; uniqueness is clear since the linear
polynomials generate $\mathbb C[z_1,\dots z_n]$ as a (unital) algebra,
and $\hat{\sigma}$ extends uniquely to the completion $UC(E)$.
\end{proof}
The observations in the proof of Proposition~\ref{P:universalSA} may
be organized slightly differently.  Restricting the operator algebra
structure of $UC(E)$ to the linear polynomials, we get a completely
isometric copy of $E$.  Thus a homomorphism $\pi$ from the polynomials
into $B(K)$ is completely contractive for $UC(E)$ if and only if its
restriction to the linear polynomials is completely contractive for
the induced operator space structure.  This gives a way of detecting
whether or not a given operator algebra structure on the polynomials
agrees with some $UC(E)$.  This observation is exploited in the next
section to show that the tridisk algebra $\mathcal A(\mathbb D^3)$ is
not completely isometric to any $UC(E)$.

A routine categorical argument shows that the universal property of Proposition~\ref{P:universalSA} characterizes $UC(E)$ (up to complete isometry) among the commutative operator algebras which contain $E$ completely isometrically.  We then obtain:  
\begin{prop}  Let $E$ be a finite-dimensional operator space, $OA_u(E)$ the universal (unital) operator algebra over $E$, and $\mathcal C$ the commutator ideal of $OA_u(E)$.  Then
\begin{equation*}
UC(E) \cong OA_u(E)/\mathcal C,
\end{equation*}
completely isometrically.
\end{prop}
\begin{proof}  We begin with the observation that the map of $E$ into $OA_u(E)/\mathcal C$ given by the composition
\begin{equation*}
E\hookrightarrow OA_u(E) \to OA_u(E)/\mathcal C
\end{equation*}
is completely isometric.  (The first map is the canonical (completely isometric) embedding into $OA_u(E)$; the second is the quotient map.)  To see this, it suffices to see that the restriction of the quotient map to $E$ is completely isometric; this in turn follows from the existence of a completely isometric map $\sigma :E\to B(H)$ with commutative range.  Such a map can be obtained by taking any complete isometry $\tau:E\to B(K)$ and defining
\begin{equation*}
\sigma =\begin{pmatrix} 0 & \tau \\
                        0 & 0\end{pmatrix}.
\end{equation*}
With this canonical embedding of $E$ into the quotient in hand, it is straightforward to check that $OA_u(E)/\mathcal C$ has the universal property of Proposition~\ref{P:universalSA}, and hence is completely isometrically isomorphic to $UC(E)$.  
\end{proof}

It is shown in \cite[Chapter 6]{pisier} that the operator algebra norm on $OA(E)$ is realized by taking the supremum over just those completely contractive representations of $OA(E)$ on finite-dimensional Hilbert spaces.  It is not obvious that the same is true for $UC(E)$---the difficulty is that if $\sigma:E\to B(H)$ has commuting range and $P$ is a projection in $B(H)$, the map $P\sigma P$ need not have commuting range.   However the proof of Theorem~\ref{T:main} shows that  $UC(E)$ is indeed determined by its finite-dimensional representations:
\begin{thm}  For every matrix-valued polynomial $p$, we have
\begin{equation}\label{E:just_matrices}
{\|p\|}_{UC(E)} =\sup\|p(S)\|
\end{equation} 
where the supremum is take over all $n$-tuples of commuting \emph{matrices} for which $\sigma_S$ is completely contractive for $E$; in other words, over all commuting matrices in the unit ball of $E^*$.   
\end{thm}
\begin{proof}    This is really an immediate consequence of the fact that the operators $S_k$ constructed in the proof of the ``(3) implies (1)'' implication of Theorem~\ref{T:main} act on a finite-dimensional Hilbert space.  More precisely, (recalling the terminology and notation used in the proof of Theorem~\ref{T:main}), if $p$ is given and does \emph{not} admit a Nevanlinna factorization, then there exists a finite set $\Lambda\subset \Omega$ for which $1-p(z)p(w)^*$ does not belong to the cone $\mathcal C$.  In this setting, the proof of the (3)$\implies$(1) implication produces an $n$-tuple of operators $S=(S_1,\dots S_n)$ on the finite-dimensional Hilbert space $\mathcal H$ for which $I-p(S)p(S)^*$ is non-positive.  The contrapositive of this statement is that if $\|p(S)\|\leq 1$ for all admissible \emph{matrices} $S$, then $p$ admits a Nevanlinna factorization, and hence (\ref{E:just_matrices}) holds.
\end{proof}

Another useful fact about $OA(E)$ is that it ``commutes'' with Calderon interpolation \cite[Section 2.7]{pisier}, that is, for any pair of compatible operator spaces $E_0, E_1$, 
\begin{equation*}
OA(E_\theta)={[OA(E_0), OA(E_1) ]}_\theta
\end{equation*}
completely isometrically.  We do not know if the analogous statement is true for $UC(E)$.
\begin{ques}Is it the case that
\begin{equation*}
UC(E_\theta) \cong {[UC(E_0), UC(E_1)]}_\theta
\end{equation*}
completely isometrically?
\end{ques}
\section{Examples}\label{S:examples}

Lacking a better name, in what follows we shall refer to the operator algebra norms described by Theorem~\ref{T:main} generically as {\em $UC$-norms}.  One natural class of examples in the present context are those coming from the minimal and maximal operator space structures over the given Banach space $V$.  We briefly recall the definitions.  To define $MIN(V)$, we observe that the duality between $V$ and $V^*$ induces a natural map $e$ from $V$ into the space of continuous functions on the unit ball of $V^*$ (denoted $C(V^*_1)$), by sending $z$ to the functional it induces on $V^*$:  
\begin{equation*}
z\to \langle \cdot, z\rangle
\end{equation*}
By the Hahn-Banach theorem, this map is isometric if $C(V^*_1)$ is equipped with the supremum norm.  Since this norm makes $C(V^*_1)$ into a C*-algebra, the embedding thus defines an operator space structure on $V$, called the \emph{minimal operator space} over $V$, and is denoted $MIN(V)$.  The operator space $MAX(V)$ is defined by the matrix norms
\begin{equation*}
\|(v_{ij})\|_N:=\sup_\varphi \| (\varphi(v_{ij})\|_{B(H^N)}
\end{equation*} 
where the supremum is taken over all contractive linear maps from $V$ into $B(H)$.  In other words, every contractive map out of $V$ is completely contractive for  $MAX(V)$.  On the other hand, a map is completely contractive for $MIN(V)$ if and only if it is completely contractive for every operator space structure over $V$.  It is well-known (and not too hard to prove) that $MIN(V)^*=MAX(V^*)$ and $MAX(V^*)=MIN(V)$.  

It follows that for each $V$, there is a unique minimal and maximal $UC$-norm associated to the domain $\Omega=ball(V)$.  We denote these norms $\|p\|_{MIN(\Omega)}$ and $\|p\|_{MAX(\Omega)}$ respectively.  The largest norm has the smallest unit ball; and hence allows the fewest completely contractive maps to appear in the Nevanlinna factorization.  This happens when we choose $E=MIN(V)$ in Theorem~\ref{T:main}, so the maximal $UC$-norm over $\Omega=ball(V)$ is obtained by taking the supremum over all commuting $T$ such that the map $\sigma_T$ is completely contractive for $MIN(V)^*=MAX(V^*)$. 
For example, if $\Omega$ is the unit polydisk $\mathbb D^n$, then $V=\ell^\infty_n$ and $V^*=\ell^1_n$.  Now, $\sigma_T$ is completely contractive for $MAX(\ell^1_n)$ if and only if it is contractive, that is if and only if
\begin{equation*}
\|\sum_{j=1}^n z_j T_j\| \leq \sum_{j=1}^n|z_j|.
\end{equation*}
for all $z\in\mathbb C^n$.  Clearly this happens if and only if each $T_j$ is contractive, so by the von Neumann inequality of Theorem~\ref{T:main} we see that the maximal $UC$-norm over the polydisk is equal to the universal norm (the supremum over all commuting contractions) discussed in the introduction.  
\subsection{$MIN(\ell^1_n)$}
In fact, the above considerations allow us to observe a stronger consequence of the Kaiser-Varopoulos counterexample to the three-variable von Neumann inequality.  The original example, interpreted in the present setting, shows that $\|p\|_{MAX(\mathbb D^3)}>\|p\|_\infty$ on the polydisk $\mathbb D^3$.  In fact their example shows that $\|p\|_{MIN(\mathbb D^3)}>\|p\|_\infty$.  More precisely, the triple commuting contractions $T$ of the Kaijser-Varopoulos example \cite{var} are such that $\sigma_T$ is completely contractive for $MIN(\ell^1)$, and hence $\|p\|_{MIN(\mathbb D^3)} \geq \|p(T)\| >\|p\|_\infty$.    It should be stressed that this is a particular feature of this example and not true generically of counterexamples to the three-variable von Neumann inequality; in particular it is not true of the $8\times 8$ example produced by Crabb and Davie \cite{crabb-davie}.

\begin{prop}
The Kaiser-Varopoulos contractions are completely contractive for $\text{MIN}(\ell^1_n)$.  
\end{prop}
\begin{proof}
Let $e_1, \dots e_5$ denote the standard basis of $\mathbb C^5$.  Consider the unit vectors
\begin{align*}
v_1 &=\frac{1}{\sqrt{3}}(-e_2+e_3+e_4)\\
v_2 &=\frac{1}{\sqrt{3}}(e_2-e_3+e_4)\\
v_3 &=\frac{1}{\sqrt{3}}(e_2+e_3-e_4)\\
\end{align*}
The Kaijser-Varopoulos contractions are the commuting $5\times 5$ matrices $T_1, T_2, T_3$ defined by
$$
T_j = e_{j+1} \otimes e_1 +e_5\otimes v_j
$$
To prove the proposition we must show that if $A_1, A_2, A_3$ are matrices which satisfy
\begin{equation}\label{E:minell1def}
\|z_1 A_1 +z_2 A_2 +z_3 A_3\|\leq 1
\end{equation}
for all $z\in\mathbb D^n$, then $\|\sum A_j \otimes T_j\|\leq 1$.  Computing, we find
\begin{equation}\label{E:atensort}
A_1\otimes T_1 + A_2\otimes T_2+ A_3\otimes T_3 =\begin{pmatrix}
0 & 0 & 0 & 0 & 0 \\
A_1 & 0 & 0 & 0 & 0 \\
A_2 & 0 & 0 & 0 & 0 \\
A_3 & 0 & 0 & 0 & 0 \\
0 & B_1 & B_2 & B_3 & 0 \end{pmatrix}
\end{equation}
where 
\begin{align*}
B_1 &=\frac{1}{\sqrt{3}}(-A_1+A_2+A_3)\\
B_2 &=\frac{1}{\sqrt{3}}(A_1-A_2+A_3)\\
B_3 &=\frac{1}{\sqrt{3}}(A_1+A_2-A_3)\\
\end{align*}
The norm of the matrix (\ref{E:atensort}) is equal to the maximum of the norms of the first column and the last row.  By (\ref{E:minell1def}), we have $\|\pm A_1\pm A_2\pm A_3\|\leq 1$ for any choices of signs, so the last row of (\ref{E:atensort}) has norm at most $1$.  To say that the first column has norm at most 1 amounts to saying that 
\begin{equation}\label{E:column-contraction}
I-\sum_j A_j^* A_j \geq 0
\end{equation}
or, in fancier language, the identity map of $\mathbb C^n$ is completely contractive from $\text{MIN}(\ell^1_n)$ to the column operator space $C_n$.  This may be seen by averaging:  by (\ref{E:minell1def}), the matrix valued function
$$
I -\sum z_i\overline{z_j} A_j^* A_i 
$$
is positive semidefinite on $\mathbb T^n$.  Integrating against Lebesgue measure on $\mathbb T^n$ gives (\ref{E:column-contraction}).
\end{proof}

\subsection{$MIN(\ell^2_n)$}
We next consider the unit ball of $\mathbb C^n$, $n\geq 2$, with the $\ell^2$ norm.  Recall that the {\em row operator space} $R_n$ and {\em column operator space} $C_n$ are defined by embedding $\mathbb C^n$ into $M_n(\mathbb C)$ ``along the first row'' or ``along the first column'' respectively.  We have $R_n^*=C_n$ completely isometrically, and thus a polynomial belongs to the unit ball of $UC(C_n)$ if and only if it is contractive when evaluated on every row contraction, that is, if and only if it is a contractive multiplier of the {\em Drury-Arveson space}; this fact is Arveson's von Neumann inequality for row contractions \cite{arv}.

It is known in general that $\|p\|_{UC(C_n)} >\|p\|_\infty$ (here $\|p\|_\infty$ is the sup norm over $\mathbb B^n$); probably the simplest example is $p(z_1,z_2)=2z_1z_2$.  The next example shows that the strict inequality persists if we replace the sup norm with the $MIN(\ell^2_2)$ norm.
\begin{prop}\label{P:l2example}
Let $p(z_1,z_2)=2z_1z_2$.  Then $\|p\|_{MIN(\mathbb B^2)} =\|p\|_\infty =1$ (so in particular $\|p\|_{UC(C_2)}>\|p\|_{MIN(\mathbb B^2)}$.
\end{prop}
\begin{proof}
By Theorem~\ref{T:main} and the discussion at the beginning of this section, it suffices to exhibit a contractive map $\sigma:\ell^2_2 \to B(H)$ and a holomorphic function $F:\mathbb B^2\to H$ such that
\begin{equation}\label{E:2zwfactor}
1-4z_1z_2\overline{w_1}\overline{w_2} =F(z)(1-\sigma(z)\sigma(w)^*)F((w)^*.
\end{equation}
To do this , take $H=\mathbb \ell^2_6$; define
\begin{equation}
\sigma(z_1,z_2)= \begin{pmatrix} 0 & z_1 & z_ 2 & 0 & 0 & 0 \\
                               z_2 & 0 & 0 & 0 & 0 & 0 \\
			       z_1 & 0 & 0 & 0 & 0 & 0 \\
			       0 & 0 & 0 & 0 & z_1 & z_2 \\
			       0 & 0 & 0 & z_2 & 0 & 0 \\
			       0 & 0 & 0 & z_1 & 0 & 0 \end{pmatrix}
\end{equation}
and 
\begin{equation}
F(z_1, z_2)= \begin{pmatrix} 1 & 0 & 0 & 0 & z_1 & z_2 \end{pmatrix}
\end{equation}
Clearly $\|\sigma(z_1,z_2)\|=|z_1|^2+|z_2|^2$, and one may then check that (\ref{E:2zwfactor}) holds.
\end{proof}
Actually, what is most interesting about this example is not that $\|p\|_{R_2}>\|p\|_{MIN(\mathbb B^2)}$ but that $\|p\|_{MIN(\mathbb B^2)}=\|p\|_\infty$.  Similar factorizations can be constructed in higher dimensions to show that $\|p\|_{MIN(\mathbb B^n)}=1$ for the polynomials
\begin{equation}
p(z_1, \dots z_n) = n^{n/2} z_1z_2\cdots z_n, \quad p(z_1, \dots z_n) =z_1^2 +\cdots +z_n^2.
\end{equation}
This is interesting because these are polynomials satisfying $\|p\|_\infty=1$ that are in some sense ``large''; in particular their Cayley transforms are extreme points of the space of holomorphic functions with positive real part in $\mathbb B^n$ \cite[Section 19.2]{Rud}.  It is then natural to raise the following question, which seems quite difficult:
\begin{ques}
Is it the case that $\|p\|_{MIN(\mathbb B^n)}=\|p\|_\infty$ for all polynomials $p$?
\end{ques}  
In the language of \cite[Chapter 5]{paulsen}, the algebra of polynomials in $n$ variables with the operator algebra norm defined by taking the supremum over all commuting contractions is denoted $(\mathcal P_n, \|\cdot\|_u)$.  Paulsen also considers the algebras $\mathcal A(\mathbb D^n)$ and $MAXA(\mathcal A(\mathbb D^n)$.  The former is the operator algebra determined by the supremum norm, the latter is obtained by taking the supremum over all commuting contractions which satisfy von Neumann's inequality.  In our notation (up to taking closures) the algebra $(\mathcal P_n,\|\cdot\|_u)$ is $UC(MAX(\ell^1_n))$.  The above example shows that the norm on $UC(MIN(\ell^1_n))$ strictly dominates the sup norm when $n\geq 3$; it follows that none of the algebras 
\begin{gather*}
UC(MAX(\ell^1_n)), \\
UC(MIN(\ell^1_n)), \\
MAXA(\mathcal A(\mathbb D^n),\\
\mathcal A(\mathbb D^n)
\end{gather*}
are completely isometrically isomorphic to each other when $n\geq 3$.  However it is an open problem to determine if any of these are pairwise completely boundedly isomorphic.  By definition chasing, the identity map on polynomials induces complete contractions
\begin{equation*}
UC(MAX(\ell^1_n)) \to UC(MIN(\ell^1_n)) \to \mathcal A(\mathbb D^n)
\end{equation*}
and 
\begin{equation*}
UC(MAX(\ell^1_n)) \to MAXA(\mathcal A(\mathbb D^n) \to \mathcal A(\mathbb D^n)
\end{equation*}
The relationship between $UC(MIN(\ell^1_n))$ and $MAXA(\mathcal A(\mathbb D^n)$ is less clear; each possesses completely contractive maps into $B(H)$ which are not completely contractive for the other.  

%We will resist the temptation to consider the operator algebra
%\begin{equation*}
%MAXA(UC(MIN(\ell^1_n))).
%\end{equation*}

\section{Further results}

One may view the presence of only polynomials in Theorem~\ref{T:main}
as too restrictive, but the statement admits a simple modification to
make it valid for arbitrary analytic functions on $\Omega$.  All that
is required is to restrict the von Neumann inequality to strictly
completely contractive tuples $S$; that is, $S$ for which
$\|\sigma_S\|_{cb}=r<1.$  It is not hard to see that this condition
implies that the Taylor spectrum of $S$ lies in the closure of
$r\Omega$, and hence $f(S)$ is a well-defined, bounded operator for
any $f$ holomorphic in $\Omega$.  We then have:

\begin{thm}\label{T:main_redux}  Let $V$ be an $n$-dimensional Banach space, $E$ an operator space structure over $V$, and $\Omega=\text{ball}(V)\subset \mathbb C^n$.  For every function $f$ holomorphic in $\Omega$, the following are equivalent:

\begin{itemize}
\item[{\bf 1)}] {\textbf{\emph{Agler-Nevanlinna factorization.}}}
  There exists a Hilbert space $K$, a completely contractive map
  $\psi:V\to B(K)$, and an analytic function $F:\Omega \to B(K,\mathbb
  C^N)$ such that
\begin{equation}\label{E:factorization_redux}
1-f(z)f(w)^* = F(z)\left[ I_K -\sigma(z)\sigma(w)^* \right]F(w)^*
\end{equation}

\item[{\bf 2)}] {\textbf{\emph{Transfer function realization.}}}
  There exists a Hilbert space $K^\prime$, a unitary transformation
  $U: K^\prime\oplus \mathbb C^N \to K^\prime\oplus \mathbb C^N$ of
  the form
\begin{equation}
\begin{array}{cc}
		\	& 	\begin{array}{cc}
		 		K^\prime & \mathbb C^N
		 		\end{array}\\
		\begin{array}{c}
			K^\prime \\
			\mathbb C^N
		\end{array}  & 	\left( \begin{array}{ll}
						A & B \\
						C & D
				       \end{array} \right)
	\end{array}
\end{equation}
and a completely contractive map $\sigma:V\to B(K^\prime)$ so that
\begin{equation}\label{E:transfer_redux}
f(z) = D + C(I-\sigma(z)A)^{-1}\sigma(z)B.
\end{equation}
\item[{\bf 3)}]  {\textbf{\emph{von Neumann inequality.}}}  If $S$ is a commuting $n$-tuple in $B(K)$ and $\sigma_S$ is {\em strictly} completely contractive for $E^*$ (that is, $\|\sigma_S\|_{cb}<1$), then
\begin{equation}\label{E:vonneumann_redux}
{\|f(S)\|}_{M_N\otimes B(K)} \leq 1.
\end{equation}
\end{itemize}
\end{thm} 
\begin{proof}[Proof (sketch)]  The ``1 implies 2'' and ``2 implies 3'' proofs are essentially unchanged.  For ``3 implies 1,'' fix the finite set $\Lambda$ and the cone $\mathcal C$ as in the original proof.  Since $\mathcal C$ is closed and $I_N-f(\lambda_i)f(\lambda_j)^*$ is assumed to be outside of $\mathcal C$, there exists $0<r<1$ so that $I_N-f_r(\lambda_i)f_r(\lambda_j)^*$ is still outside of $\mathcal C$, where $f_r(z):=f(rz)$.  Now continue the proof as before with $f_r$ in place of $f$.  The GNS construction produces, as in the original proof, operators $S_i$ so that $\sigma_S$ is completely contractive for $E^*$.  Finishing the proof shows that 
\begin{align*}
\langle (I-f(rS)f(rS)^*) G,G\rangle_{\mathcal H^\prime} &=\langle
(I-f_r(S)f_r(S)^*) G,G\rangle_{\mathcal H^\prime} \\ &=L(I_N -
f_r(\lambda_i)f_r(\lambda_j))\\ &<0.
\end{align*} 
The operators $rS_i$ thus give a strictly completely contractive map
for $E^*$ and the desired contradiction.
\end{proof}

As is now well-understood, the equivalences in
Theorem~\ref{T:main_redux} also give rise to a Nevanlinna-Pick
interpolation theorem for the Banach algebra of holomorphic functions
on $\Omega$ with the norm whose unit ball is characterized by
Theorem~\ref{T:main_redux}.  (Extending the notation of the previous
section, we will call this algebra $UC^\infty(E)$).  We state here
only the most elementary scalar version; by well-known techniques the
result may be extended to cover matrix-valued interpolation.

\begin{thm}\label{T:NP}  Given points $\lambda_1, \dots \lambda_N$ in $\Omega$ and scalars $w_1, \dots w_N$, there exists a function $f\in UC^\infty(E)$ satisfying $f(\lambda_j)=w_j$ for all $j=1, \dots N$ if and only if there exist matrices $T_1, \dots T_n$ such that $\sigma_T$ is completely contractive for $E$ and vectors $v_1, \dots v_N$ such that
\begin{equation}\label{E:NP}
1-w_i\overline{w_j} = v_i [I-\sum_{k,l=1}^n \lambda_i^k\overline{\lambda_j^l}T_k T_l^*]v_j^*
\end{equation}   
\end{thm}
\begin{proof}  If $f\in UC^\infty(E)$ and $f(\lambda_j)=w_j$, then (\ref{E:NP}) is simply the restriction of (\ref{E:factorization_redux}) to the points $\lambda_1, \dots \lambda_N$, with $T_k=\sigma(e_k)$.  Conversely, if (\ref{E:NP}) holds, we set $\sigma=\sigma_T$ and run the lurking isometry argument; this produces a unitary colligation such that
\begin{equation}\label{W:NP_transfer}
w_j =D+C(I-\sigma(\lambda_j)A)^{-1}\sigma(\lambda_j)B.
\end{equation}
But this transfer function realization extends to define a function $f$ in all of $\Omega$, and by Theorem~\ref{T:main_redux} this $f$ lies in $UC^\infty(E)$.  
\end{proof}

\bibliographystyle{plain} 
\bibliography{optestfns29sep}

\end{document}